\documentclass[11pt]{amsart}

\usepackage{euscript}
\usepackage{amsmath}
\usepackage{amsthm}
\usepackage{epsfig}
\usepackage{amssymb}\usepackage{amscd}
\usepackage{epic}

\theoremstyle{plain}
\newtheorem{theorem}[equation]{Theorem}
\newtheorem{lemma}[equation]{Lemma}
\newtheorem{proposition}[equation]{Proposition}

\newtheorem{conjecture}[equation]{Conjecture}

\theoremstyle{definition}
\newtheorem{example}[equation]{Example}

\newtheorem{known}[equation]{Known results}

\def\C{\mathbb C}
\def\Q{\mathbb Q}
\def\R{\mathbb R}
\def\Z{\mathbb Z}

\newcommand{\labelpar}{\label}
\newcommand{\m}{\mathfrak{m}}

\def\?{{\bf ???}}

\newcommand{\tX}{\tilde{X}}
\newcommand{\cO}{{\mathcal O}}

\def\Z{{\mathbb Z}}
\def\R{{\mathbb R}}
\def\fa{\mathfrak{a}}
\def\fm{\mathfrak{m}}
\def\bu{{\bf u}}
\def\bv{{\bf v}}
\def\bx{{\bf x}}
\def\r{\rho}
\def\^{\widehat}
\def\({\left(}
\def\){\right)}

\DeclareMathOperator{\Vol} {Vol}
\DeclareMathOperator{\ini} {in}

\title{Durfee's conjecture on the signature of \\ smoothings of surface singularities}
\author{J\'anos Koll\'ar and Andr\'as N\'emethi\\
{\lowercase{ with an appendix by}} \\
Tommaso de Fernex}

\date{}

\begin{document}
\begin{abstract} In 1978 Durfee conjectured
various   inequalities
between the signature $\sigma$ and the geometric genus $p_g$ of a normal
surface singularity. Since then a few counter examples have been found
and positive results established in some special cases.

We prove a `strong' Durfee--type
inequality for any smoothing of a Gorenstein  singularity,
provided  that the  intersection form the resolution is unimodular,
and the conjectured `weak' inequality for  all hypersurface singularities
and for sufficiently large
multiplicity strict complete intersections.
The proofs   establish general
inequalities valid for any normal surface singularity.
\end{abstract}

\maketitle


\pagestyle{myheadings} \markboth{{\normalsize
J. Koll\'ar and A. N\'emethi}}{ {\normalsize Durfee's Conjecture}}

\section{Introduction}\labelpar{s:i}

{\bf Durfee's conjectures.}
Let $(X,0)$ be a complex analytic normal surface singularity
and  $\tX\to X$  a resolution. The {\it geometric genus}
 $p_g$ is defined as
$h^1(\cO_{\tX})$. For any one--parameter smoothing with generic (Milnor)
 fiber $F$, the rank of
the second homology  $H_2(F,\Z)$ is the {\it Milnor number}  of the smoothing
$\mu$. Furthermore,  $H_2(F,\Z)$
  has a natural intersection form with Sylvester invariants
 $(\mu_+,\mu_0,\mu_-)$. Then $\mu=\mu_++\mu_0+\mu_-$ and  $\sigma:=\mu_+-\mu_-$
 is called the {\it signature} of the smoothing.
The Milnor number and the signature  usually depend on the
 choice of the smoothing.
For more details  see the monographs \cite{AGV,AGLV,Looijenga-book,Milnor-book}
or \cite{Laufer1977,Looijenga1986,Wahl1981}.
Formulas for various classes of singularities can be found in
\cite{Greuel1975,Greuel-Hamm1978,Hamm1986,Hamm2011,Kerner-Nemethi.a,Kerner-Nemethi.b,Khovanskii1978,Morales1985}.

These local invariants should be viewed as analogs of the
most important global invariants: Todd genus,  Euler number and signature.

Durfee proved that    $2p_g=\mu_0+\mu_+$ \cite{Durfee1978} and
 $\mu_0$ equals  the first Betti number $b_1(L_X)$ of the link $L_X$ of $(X,0)$.

Examples show that for  a  surface singularity $\mu_-$ is
usually large compared to the
other Sylvester invariants. Equivalently, $p_g$ is essentially  smaller  than
 $\mu$ and $\sigma$ tends to be rather negative.
These observations led to the formulation of  Durfee's Conjectures
  \cite{Durfee1978}.

\vspace{2mm}

\noindent \
 {\bf Strong inequality}: If $(X,0)$ is an isolated complete intersection surface singularity (ICIS)
 then $6p_g\leq \mu$.\\
 \ {\bf Weak inequality}: If $(X,0)$ is a normal surface singularity, then for any smoothing
$4p_g\leq \mu+\mu_0 $. Equivalently, $\sigma\leq 0$.\\
\ {\bf Semicontinuity of $\sigma$}: If $\{(X_t,0)\}_{t\in (\C,0)}$ is a flat
family of
isolated surface singularities then $\sigma(X_{t=0})\leq \sigma(X_{t\not=0})$.
\vspace{2mm}

Other  invariants  are
provided by the combinatorics of a resolution $\pi: \tX\to X$.
Let $s$ denote the number
of irreducible $\pi$-exceptional curves  and $K$ the canonical
class of
$\tX$. Then $K^2+s$ is independent of the resolution and, for Gorenstein singularities,
\begin{equation}\label{eq:Laufer}
\mu=12p_g+K^2+s-\mu_0\quad\mbox{and}\quad -\sigma=8p_g+K^2+s;
\end{equation}
see \cite{Durfee1978,Laufer1977,Steenbrink1983,Wahl1981}.
Therefore, an inequality of type $\mu+\mu_0\geq C\cdot p_g$ (for some constant $C$)
transforms into $(12-C)p_g+K^2+s\geq 0$, or $-\sigma\geq (C-4)p_g$.

The resolution defines the
 {\it minimal cycle} $Z_{min}$
(also called the {\it Artin} or {\it fundamental} cycle) and the
 {\it maximal cycle}  $Z_{max}$.
The former is the smallest integral effective cycle  $\sum e_iE_i$
such that $(E_j,\sum e_iE_i)\leq 0$ for every $\pi$-exceptional curve
$E_j\subset \tX$ and the latter is the divisor corresponding to the ideal sheaf
$\pi^{-1}\m_{X,0}\cdot \cO_{\tX}$. It is clear that $Z_{min}\leq Z_{max}$.

Other invariants of $(X,0)$ are
the {\it multiplicity,} denoted by  $\nu$,
and  the {\it embedding dimension,} denoted by $e$.

 \begin{known}\label{1.4}
A counterexample to the {\it weak inequality}  was given by Wahl
\cite[page 240]{Wahl1981}; it is a minimally  elliptic
normal surface singularity (not ICIS) with  $\nu= 12$,
$\mu=3$, $\mu_0=0$, $p_g=1$ and $\sigma=1$.
Nevertheless,  both the strong and the weak inequalities hold
in most examples and
the intrinsic structure
responsible for the positivity/negativity of the signature of a given germ
has not been understood.

A counterexample to the {\it semicontinuity} of the signature was found in
\cite{Kerner-Nemethi2009};
 this excludes  degeneration arguments in
  possible proofs of the  inequalities.

The articles \cite{Kerner-Nemethi.a,Kerner-Nemethi.b}
 show that the {\it strong inequality} also fails
for some   non--hypersurface
 ICIS, and without other restrictions
the best that we can expect is the  weak inequality.

For hypersurfaces we have the following `positive' results:

\vspace{2mm}

$8p_g<\mu$ for $(X,0)$ of multiplicity 2, Tomari \cite{Tomari93},

$6p_g\le\mu-2$ for $(X,0)$ of multiplicity 3, Ashikaga \cite{Ashikaga92},

$6p_g\le\mu-\nu+1$ for quasi-homogeneous singularities, Xu--Yau \cite{Xu-Yau1993},

$6p_g\le\mu$ for suspension singularities $\{g(x,y)+z^k=0\}$, N\'emethi \cite{Nemethi98,Nemethi99},

$6p_g\le\mu$ for absolutely isolated singularities, Melle--Hern\'andez \cite{Melle00}.

For a short proof of $\sigma\leq 0$ in the suspension case see \cite{Nemethi1998}.
\end{known}


In this note we  estimate the expression $8p_g+K^2+s$
using  properties of the dual graph of the minimal  resolution.
For smoothable  Gorenstein singularities we obtain the following.

\begin{theorem}\label{main.thm}
  Let $(X,0)$ be a normal  Gorenstein  surface singularity
with embedding dimension $e$ and geometric genus $p_g$.
Let $\sigma$ denote the signature of a smoothing. Then
\begin{enumerate}
\item  If  the resolution intersection form is unimodular then
$-\sigma\geq 2^{4-e}(p_g+1)$.
\item If $(X,0)$ is a (non smooth) hypersurface  singularity then $-\sigma\geq 1+\mu_0$.
\end{enumerate}
\end{theorem}

We prove  several inequalities that hold
without the  Gorenstein assumption. At each step we `loose
something'. Analyzing these steps should lead to better
estimates in many cases.
 Our aim is  not to  over-exploit these technicalities,
but to show conceptually the
general principles behind the inequalities.

It seems that $-\sigma\geq 0$ for all `sufficiently complicated'
complete intersections, but we can prove this only
for {\it strict complete intersection} singularities where
a local ring $({\mathcal O}_{X,0},\m_{X,0})$ is called a
 strict complete intersection iff
the corresponding graded ring ${\rm Gr}_{\m_{X,0}}({\mathcal O}_{X,0})$
is a complete intersection; see \cite{bennett}.

\begin{proposition} Fix $e$ and consider
the set of  strict  ICIS of embedding dimension $e$. Then
$-\sigma$ tends to infinity whenever
the multiplicity $\nu$ tends to infinity.
\end{proposition}

\begin{example}\label{ex:hom}\cite{Kerner-Nemethi.a,Kerner-Nemethi.b}
 Assume that $(X,0)$ is a homogeneous ICIS of codimension $r=e-2$ and
multidegree  $(d,\ldots, d)$.
Then $\nu=d^r$ and
$$\frac{p_g}{\nu}=\frac{r(d-1)(d-2)}{6}+\frac{r(r-1)(d-1)^2}{8}; \ \
\frac{\mu+1-\nu}{\nu}=r(d^2-3d+2)+\frac{r(r-1)(d-1)^2}{2}.$$

\vspace{1mm}

(a) \ If $r=1$ then $6p_g = \mu+1-\nu$.

(b) If $r\geq 2$ is fixed then $\frac{\mu}{p_g}$ asymptotically
 tends to $C_{2,r}:=\frac{4(r+1)}{r+1/3}$, although
  $C_{2,r}\cdot  p_g\leq \mu+1$ does not hold in general.
(The constant 4 is the best  bound valid for any $d$.)

(c) For any $r$ the inequality $4p_g\leq \mu+1-\nu$ is valid. In fact, for any fixed $d$
$$4\cdot\frac{(d-1)(r-1)+2(d-2)} {(d-1) (r-1)+\frac{4}{3}(d-2)}\cdot p_g\leq \mu+1-\nu.$$
For $d=2$ the coefficient of $p_g$ is 4, this coefficient is increasing if $d$ is increasing.
\end{example}

  {\bf Acknowledgments.}
 We thank M.~Musta\c{t}\u{a}  for useful  suggestions.
Partial financial support  to JK  was provided  by  the NSF under grant numbers
DMS-07-58275 and DMS-13-62960.
Partial financial support  to AN  was provided  by OTKA Grants 81203 and 100796.
This paper was written  while AN  visited Princeton University.

\section{Minimal Euler characteristic of  a resolution}\label{s:2}

Let $(X,0)$ be a normal surface singularity with
 minimal resolution $\tX\to X$. We write $L=H_2(\tX,\Z)$,
$(\cdot, \cdot)$ denotes  the intersection form
on $L$ and  $L'$ is the dual lattice ${\rm Hom}_\Z(L,\Z)$ with
natural inclusions $L\subset L'\subset L\otimes \Q$.

Let $Z_K\in L'$ be the anticanonical cycle, that is,
$(Z_K, E_i)=-(K, E_i)$ for every exceptional curve $E_i$.
 By  the minimality of the resolution
$(Z_K,l)\leq 0$ for any effective rational cycle $l$ and $Z_K\geq 0$.
A singularity is called {\it numerically Gorenstein}
if $Z_K\in L$.

Set $\chi(l')=-(l',l'-Z_K)/2$ for any $l'\in L\otimes\Q$.
By Riemann--Roch and the adjunction formula,
  $\chi(l)=\chi(\cO_l)$
for any effective cycle  $l\in L$.
We set
$$
\min\chi:=\min_{l\in L}\chi(l).
$$
 It is a topological invariant of $(X,0)$,
strongly related to arithmetical properties of the lattice $L$.
Usually it is hard to compute explicitly.
In the literature $1-\min\chi=p_a$ is called the {\it arithmetic genus } of $(X,0)$ \cite{Wael}.

(The expression $\min\chi$ is also the normalization term of the
Seiberg--Witten invariant of the link
expressed in terms of the lattice cohomology \cite{NJEMS}. The
comparison of $\min\chi$
with the $d$--invariant of the link provided by the Heegaard--Floer theory
and the
involved topological inequalities lead the authors to the ideas of the
present note.)

The quantity $\min\chi$   satisfies two obvious
inequalities. Since $h^0(\cO_l)-h^1(\cO_l)\geq 1-p_g$ we get
$\min\chi\geq 1-p_g$. Also, since the real quadratic function
$\chi(x)=-(x,x-Z_K)/2$ has its minimum
at $Z_K/2$, and $\chi(Z_K/2)=K^2/8$, we get that $\min\chi\geq K^2/8$.

We wish to understand how sharp these
inequalities are. The first inequality $\min\chi\geq 1-p_g$ will be improved to
 $\min\chi\geq -Cp_g$ for a certain constant  $0<C<1$.
This will be applied in the form $p_g+\chi(l)\geq (1-C)p_g$ for any $l$.

On the other hand, we wish to bound the difference $\min\chi-K^2/8$
from above. The strategy is the following.
Assume that for some rational cycle $\xi$ one has  $Z_K-\xi=2l\in 2L$.
Then $\chi(l)=(K^2-\xi^2)/8$,
hence $\chi(l)-K^2/8$ is minimal  exactly when $-\xi^2/8$ is minimal
among the rational cycles $\xi$ satisfying $Z_K-\xi\in 2L$.
Thus the existence of a cycle $\xi$ with $\xi^2+s\geq 0$  implies
that  $(K^2+s)/8\geq \min \chi$,
which combined with the first inequality gives  $p_g+(K^2+s)/8\geq (1-C)p_g$.

\begin{lemma}\label{lem:comb1} Let $(X,0)$ be a
numerically Gorenstein singularity. Then
$\min\chi$ is achieved by a cycle $l\in L$ satisfying $ Z_K/2\leq l\leq Z_K$.
\end{lemma}
\begin{proof} Assume that  $\chi(l)=\min\chi$ and write
$l=a-b$, where $a,b\in L$ are effective and have no common components.
Then $\chi(a+b)-\chi(a-b)=(b,Z_K-2a)\leq 0$, thus
$\chi(a+b)\leq \chi(a-b)$. Thus we may assume the $l$ is effective.
Similarly, write $l=Z_K-a+b$. Then
$\chi(Z_K-a+b)-\chi(Z_K-a-b)=(b,2a-Z_K)\geq 0$. These two
inequalities applied repeatedly show
that the minimum is achieved for some $l\in L$ with $0\leq l\leq Z_K$.

Take such a cycle and write it as
$l=Z_K/2+a-b$, $a,b\in \frac{1}{2}L$, effective and without common components.
Then $\chi(Z_K/2+a+b)-\chi(l)=-2(a,b)\leq 0$.
\end{proof}

If $(X,0)$ is a Du~Val singularity then $Z_K=0$ hence
$\min\chi(l)$ is realized by the empty cycle  $l=0$. This tends to
mess up our formulas and we exclude them in the sequel.
If  $(X,0)$ is numerically Gorenstein but not Du~Val then the support of
$Z_K$, and hence  the support of $l\geq Z_K/2$, is  the whole
exceptional set of the resolution.

\begin{proposition}\label{prop:min} \ Set $\epsilon=1$ if $(X,0)$ is
Gorenstein, and $\epsilon=0$ otherwise.
Then for any
numerically Gorenstein,  non-Du~Val  surface singularity
$p_g+\min\chi\geq 2^{\epsilon-e}(p_g+1)$.

\end{proposition}
\begin{proof}
Fix $l\in L$ such that  $Z_K/2\leq l\leq Z_K$ and $\min\chi=\chi(l)$.
In the non-Du~Val case $Z_K>0$, hence $l>0$ too and
$$
p_g+\chi(l)= p_g-h^1(\cO_l) +h^0(\cO_l)\geq h^0(\cO_l).
$$
Note that for  any effective $m\in L$ we have
$$
h^0(\cO_m)\geq \dim \bigl(H^0(\cO_{\tX})/H^0(\cO_{\tX}(-m))\bigr).
$$
The inequality is usually strict but if $m=Z_K$ then
Grauert--Riemenschneider vanishing implies that
$$
h^0(\cO_{Z_K})= \dim \bigl(H^0(\cO_{\tX})/H^0(\cO_{\tX}(-Z_K))\bigr)=p_g.
$$
Note that $H^0(\cO_{\tX})$ equals the local ring $R$ of $(X,0)$
and each $H^0(\cO_{\tX}(-m)) $ can be identified
with an ideal sheaf  $I(m)\subset R$. This correspondence is
sub-multiplicative, that is,
$I(m_1)\cdot I(m_2)\subset I(m_1+m_2)$. Thus,
for every $m$, Lemma  \ref{lem:alg} shows that
$$
\dim \bigl(H^0(\cO_{\tX})/H^0(\cO_{\tX}(-m))\bigr)\geq 2^{-e}
(1+\dim \bigl(H^0(\cO_{\tX})/H^0(\cO_{\tX}(-2m))\bigr)).
$$
Putting these together gives that
$$
\begin{array}{rcl}
p_g+\chi(l)&\geq &
\dim \bigl(H^0(\cO_{\tX})/H^0(\cO_{\tX}(-l))\bigr)\\
&\geq & \frac{1}{2^e}\bigl(1+\dim \bigl(H^0(\cO_{\tX})/H^0(\cO_{\tX}(-2l))\bigr)\bigr)\\[1ex]
&\geq &\frac{1}{2^e}\bigl(1+\dim \bigl(H^0(\cO_{\tX})/H^0(\cO_{\tX}(-Z_K))\bigr)\bigr)\\[1ex]
 & = & \frac{1}{2^e} (p_g+1).
\end{array}
$$
Let $0\leq m\leq Z_K$ be a cycle and set $\bar m=Z_K-m$.
In the Gorenstein case
 duality gives that
$$
\begin{array}{rcl}
h^1(\cO_m)=h^0(\cO_m(-\bar{m}))&=&
\dim \bigl(H^0(\cO_{\tX}(-\bar{m}))/H^0(\cO_{\tX}(-Z_K))\bigr)\\[1ex]
& = &
p_g-\dim \bigl(H^0(\cO_{\tX})/H^0(\cO_{\tX}(-\bar{m}))\bigr),
\end{array}
$$
hence, using Lemma \ref{lem:alg} in the 3rd line we get that
\begin{equation}\label{eq:needed}
\begin{array}{rcl}
p_g+\chi(m)&=& p_g-h^1(\cO_m) + h^0(\cO_m)\\[1ex]
&\geq &
\dim \bigl(H^0(\cO_{\tX})/H^0(\cO_{\tX}(-\bar{m}))\bigr)+
\dim \bigl(H^0(\cO_{\tX})/H^0(\cO_{\tX}(-m))\bigr)\\[1ex]
&\geq &\frac{1}{2^{e-1}}\,
\bigl(1+\dim \bigl(H^0(\cO_{\tX})/H^0(\cO_{\tX}(-Z_K))\bigr)\bigr)\\
&= &   \frac{1}{2^{e-1}}\, (p_g+1).
\end{array}
\end{equation}
For $m=l$ this gives the claimed inequality.
\end{proof}

\section{Inequalities in the unimodular case.}

 Assume that the intersection form of $L$ is unimodular,
that is $L=L'$.  Note that this holds iff
 the first integral homology of the link of $(X,0)$
is torsion free since
this torsion group  is isomorphic  to $L'/L$ by \cite{mumf-top}.

\begin{theorem}\label{th:ineq}
Let $(X,0)$ be a normal  surface singularity of embedding dimension $e$.
Let $\tX\to X$ be the minimal resolution with canonical class $K$
and $s$ exceptional curves. Assume that
the  resolution intersection form is unimodular. Then
\begin{enumerate}
\item  $ (K^2+s)/8\geq \min\chi$ and
\item $p_g+(K^2+s)/8\geq 2^{\epsilon-e}\, (p_g+1)$,
equivalently,  $  (K^2+s)/8\geq -(1-2^{\epsilon-e})\, p_g+2^{\epsilon-e},$
where   $\epsilon $ is as in Proposition \ref{prop:min}.
\end{enumerate}
\end{theorem}

\begin{proof}
By a result of Elkies \cite{Elkies}, there is a  $\xi\in L$
such that $\xi^2+s\geq 0$ and   $(m, m-\xi)$ is even for every $m\in L$.
(That is, $\xi$ is a  {\it characteristic element} of small norm.)
If $E$ is an irreducible exceptional curve then
$(E, E-Z_K)=2g(E)-2$ is even, thus  $(m, m-Z_K)$ is even for every $m\in L$.
Therefore  $(m, Z_K-\xi)$ is even for every $m\in L$ and
$l:=\frac12 (Z_K-\xi)\in L$.
(We used unimodularity here and it is also needed in \cite{Elkies}.)

Then $(K^2+s)/8=\chi(l)+(\xi^2+s)/8\geq \chi(l)$ and we can
 apply Proposition \ref{prop:min}.
\end{proof}

If, in addition,  $(X,0)$ is Gorenstein, then $\epsilon=1 $
thus (2) of  Theorem \ref{th:ineq} and the second formula of
(\ref{eq:Laufer}) give that
\begin{equation}
-\sigma=8p_g+K^2+s\geq 2^{4-e}\, (p_g+1).
\end{equation}
This completes the proof of part (1) of Theorem \ref{main.thm}.\qed

\medskip
The above theorem shows that the torsionfreeness
 of the first homology of the link
has more substantial effect on the negativity of the signature than
 the embedded properties, like being a hypersurface or an
ICIS.

\begin{example}
Assume that $(X,0)$ is a hypersurface singularity with $L=L'$.
Then $-\sigma\geq 2p_g+2$, or equivalently, $\mu+\mu_0\geq 6p_g+2$.
In particular, if the link of a hypersurface singularity is an
integral homology sphere (hence $\mu_0=0$ too), then it satisfies the
strong Durfee inequality $6p_g\leq \mu-2$
with the  optimal asymptotic constant 6.
\end{example}

\section{The general case}\label{s:4}

  In this section we assume that $(X,0)$ is numerically Gorenstein
but not Du~Val.
Set $x:=2\{Z_K/2\}\in L$ and $\bar{x}:=E-x$,
where $E$ is the reduced exceptional curve. Then $m:=( Z_K-x)/2=\lfloor Z_K/2\rfloor
\in L$.
 We write  $\Sigma$ for $8p_g+K^2+s$.
(Thus, in the Gorenstein case, $\sigma=-\Sigma$.)

Since $8\chi\bigl(m)=K^2-x^2$,
by Proposition  \ref{prop:min}
\begin{equation}\label{eq:gen1}
\Sigma =8\bigl(p_g+\chi\bigl(m\bigr)\bigr)+x^2+s
\geq 2^{\epsilon+3-e}(p_g+1)+x^2+s.
\end{equation}
Similarly,
\begin{equation}\label{eq:gen2}
\Sigma =8(p_g+\chi(m+E))+(E+\bar{x})^2+s\geq 2^{\epsilon+3-e}(p_g+1)+(E+\bar{x})^2+s.
\end{equation}
Since $x=E-\bar{x}$, adding the equations (\ref{eq:gen1}) and (\ref{eq:gen2})
 gives that
\begin{equation}\label{eq:gen3}
\Sigma \geq 2^{\epsilon+3-e}(p_g+1)+E^2+\bar{x}^2+s.
\end{equation}
For each cycle $y=x,\ \bar{x}$ and $E$ write
the relation $y^2=-2\chi(y)+(y,Z_K)$ and add
the equations (\ref{eq:gen1}) and (\ref{eq:gen3}). We get that
\begin{equation}\label{eq:gen4}
\Sigma \geq 2^{\epsilon+3-e}(p_g+1)+s-\chi(x)-\chi(\bar{x})-\chi(E)+(E,Z_K).
\end{equation}
Since $x,\ \bar{x},\ E$ are reduced,
$\chi(x)+\chi(\bar{x})+\chi(E)\leq s+1-b_1(L_X)$ (since $b_1(L_X)=b_1(E)$).
 Hence  (\ref{eq:gen4})  can be rewritten as
\begin{proposition} 
$\Sigma \geq 2^{\epsilon+3-e}(p_g+1)-1+b_1(L_X)+(E,Z_K)$
where $(E,Z_K)$ also equals  $E^2+2\chi(E)$.
Furthermore,  $-1+b_1(L_X)+(E,Z_K)=E^2+\chi(\Gamma)$ where
 $\chi(\Gamma)$ is the Euler characteristic of the
topological realization of the
resolution graph $\Gamma$. \qed
\end{proposition}

Although the term $(E,Z_K)$ is negative,
in many cases (e.g. hypersurfaces, ICIS) it is much smaller  than
$p_g$. We do not have a good general estimate, but the following
argument gives a bound that
implies the negativity of the signature in several cases.

In order to simplify the notation let us denote the constant $2^{\epsilon+3-e}-1+b_1(L_X)$
by $A$. Let $Z=Z_{max}\in L$ be the maximal cycle.
Hence $Z\geq E$, which implies that $(E,Z_K)\geq (Z,Z_K)$. For any
$t\geq e-\epsilon-3$ write $(2^{t+1}Z,Z_K)$ as $(2^{t+1}Z)^2+2\chi(2^{t+1}Z)$, hence we obtain that
\begin{equation}\label{eq:gen6}
\Sigma \geq \bigl(\tfrac{1}{2^{e-\epsilon-3}}-\tfrac{1}{2^t}\bigr)p_g+
\tfrac{1}{2^t}
 \bigl(p_g+\chi(2^{t+1}Z)\bigr)+2^{t+1}Z^2+A.
\end{equation}
Then using
$Z^2\geq -\nu$ (cf. \cite{Wael}) and Proposition \ref{prop:min} we get
the following.
\begin{lemma}\label{eq:gen7} For $t\geq e-\epsilon-3$ one has
$$ 
\Sigma \geq \bigl(\tfrac{1}{2^{e-\epsilon-3}}-\tfrac{1}{2^t}+\tfrac{1}{2^{t+e-\epsilon}}\bigr)p_g
-2^{t+1}\nu+A+\tfrac{1}{2^{t+e-\epsilon}}. \qed
$$\end{lemma}
With different choices of $t$ the coefficient of $p_g$ can be arranged to be as close to
$1/2^{e-\epsilon-3}$ as we wish, but  the price is a more negative coefficient for
$\nu$.
This expression shows that for an arbitrary normal surface singularity
we should expect an inequality of the form
$$
\Sigma \geq C_1p_g-C_2\nu+C_3\quad\mbox{for some constants}\quad
 C_1, C_2>0 \mbox{ and }  C_3>-1
$$
that depend on the
embedding dimension $e$. If $\nu$ dominates
$p_g$---as  in the example of Wahl---then $\Sigma $ can be negative.
However, if $p_g$ dominates the multiplicity, then $\Sigma$ becomes positive,
as in the next examples.
\medskip

  {\bf The case of strict complete intersections.} \
By Theorem (2.17) of Bennett \cite{bennett},
every strict ICIS is a
 normally flat deformation of an isolated
 homogeneous complete intersection singularity.
(Under such deformation $p_g$ is semicontinuous and $\nu$ is constant.)

In order to prove that $-\sigma=\Sigma$ is positive for large $\nu$ and fixed $e=r+2$, 
by Lemma \ref{eq:gen7} it is enough to show that  $p_g/\nu$  tends to infinity  with
$\nu$ for homogeneous complete intersections.
 In that case, if $d_1,\ldots , d_r$  ($d_i\geq 2$) are the degrees of the defining equations, then
\begin{equation}\label{eq:jan1}
\frac{p_g}{\nu}= \sum_i\frac{(d_i-1)(d_i-2)}{6}+\sum_{i<j}\frac{(d_i-1)(d_j-1)}{4}\end{equation}
and $\nu=\prod_id_i$, cf. \cite{Kerner-Nemethi.a,Kerner-Nemethi.b}.

Note that (\ref{eq:jan1})  does not imply  the negativity of the signature for
every strict ICIS, but it gives a much stronger result asymptotically.
 This suggests that the positivity of $\Sigma$ (or, the
negativity of the signature in the presence of Gorenstein smoothing)
 is  guided
by the ratio $p_g/\nu$. This seem to be a
 general phenomenon, not specifically  related
to embedded  properties.
\medskip

  {\bf The case of hypersurfaces.} \
Assume that $e=3$, hence $\epsilon=1$ too. Our goal is to prove the negativity of the signature without
any multiplicity restriction.
The inequality (\ref{eq:gen6}) with $t=-1$ becomes
 \begin{equation}\label{eq:genh1}
\Sigma \geq 2(p_g+\chi(Z))+ 1+b_1(L_X)-\nu.
\end{equation}
Using (\ref{eq:needed}) for $m=Z$ shows that
$p_g+\chi(Z)\geq  \frac14 (p_g+1)$, thus
 \begin{equation}\label{eq:genh1-jan}
-\sigma \geq  \tfrac12 (p_g+1)+ 1+\mu_0-\nu.
\end{equation}
By semicontinuity of the geometric genus $p_g\geq \binom{\nu}{3}$,
since the geometric genus of a degree
 $\nu$ homogeneous singularity is $\binom{\nu}{3}$.
Hence we get that
 \begin{equation}\label{eq:genh1-jan2}
-\sigma \geq  \tfrac12 \tbinom{\nu}{3}-\nu+ \tfrac32+\mu_0.
\end{equation}
In the right hand side $\tfrac12 \tbinom{\nu}{3}\geq \nu $ for $\nu\geq 5$, hence 
$-\sigma \geq \tfrac32+\mu_0$.
If $\nu=2$ or 3 then $-\sigma\geq 1+\mu_0$  follows from
the  results of Tomari and Ashikaga mentioned
in Paragraph \ref{1.4}.
Finally, if $\nu=4$ then $p_g\geq \tbinom{4}{3}=4$ hence
$p_g+\chi(Z)\geq  \frac14 (p_g+1)\geq \frac54$. Since $p_g+\chi(Z)$
is an integer, it has to be $\geq 2$ thus (\ref{eq:genh1})
becomes
 \begin{equation}\label{eq:genh1-jan3}
-\sigma \geq  2\cdot 2+ 1+\mu_0- 4=1+\mu_0.
\end{equation}
This completes the proof of part (2) of Theorem \ref{main.thm}.\qed
\medskip

It is possible to analyze this case further and prove
stronger lower bounds for $-\sigma$. For instance,
for  an isolated  hypersurface singularity with $\nu\geq 4$ one
can show that
 $$
-\sigma\geq \tfrac{2}{3}\bigl(p_g-\tbinom{\nu}{3}\bigr)+
2 \tbinom{\nu-1}{3}-\nu+3+\mu_0.
$$

\section{Speculations regarding generalizations of Elkies's result}

The result of Elkies---valid for unimodular definite lattices---lies
behind the `strong' inequalities of Theorem \ref{th:ineq}.
It is somewhat surprising that comparable  inequalities
can be obtained by the alternative  methods of Section \ref{s:4}.

In this section we analyze different possibilities to extend
\cite{Elkies} to the non-unimodular case,
and the effect of such extensions on Durfee-type  inequalities.

Owens  and Strle  prove that there exists
a characteristic element $\xi\in L'$ such that $\xi^2+s\geq 0$.
Thus there exists
$l'\in L'$ such that $Z_K-\xi=2l'\in 2L'$. However, this is not really helpful
to us if
$l'\not\in L$. We need  to approximate
$l'$   with an integral cycle $l\in L$ and
 the final output is not better than the results
of  Section \ref{s:4}.

Therefore, we need a generalization of the Elkies theorem that guarantees the
existence of some
$\xi$ with $\xi^2$ not very negative, and $Z_K-\xi\in 2L$.
Examples show that in general we cannot
expect $\xi^2+s\geq 0$. (For instance,
 this would contradict  the existence of singularities with
positive signature.)

  The results of \cite{Elkies}
are  valid for any abstract lattice, and in this general context
we do not have any guess about  the right form of a weaker inequality.
However, lattices coming from
 singularities have distinguished bases and a positive cone
 of  effective divisors. Having these in mind, and also the type of
inequalities we already
 obtained, we can speculate on how to weaken the Elkies inequality
using a combinatorial object of the lattice
  related to the multiplicity of the singularity.
 Computation of several examples supports the following conjecture.

\begin{conjecture} Let $Z_{min}\in L$ be the minimal
cycle of the singularity
lattice $L$. Then there exists a cycle $\xi\in L'$, with $Z_K-\xi\in 2L$
such that
$$\xi^2+s\geq Z_{min}^2.$$
\end{conjecture}

As in the proof of Theorem \ref{th:ineq}, the conjecture would imply that,
for any normal surface singularity
$$-\sigma\geq \tfrac{1}{2^{e-\epsilon-3}}\cdot (p_g+1)-\nu.$$
For hypersurfaces this becomes  $-\sigma\geq 2(p_g+1)-\nu$.
Keeping in mind that
for hypersurfaces $p_g\geq \binom{\nu}{3}$ (that is, $\nu\leq Cp_g^{1/3}$),
this inequality is a good replacement for the
 expected strong  inequality.
 These methods would imply that $-\sigma\geq 0$ for every
large multiplicity ICIS but they fall short in general.
(However, the constants are better than those in (\ref{eq:gen7})).
We believe that  small multiplicity ICIS should be studied by techniques
 specific to them.

Nevertheless, we hope that
the above conjecture has interesting  arithmetical and geometrical meaning
and that it is  related to the topological $d$--invariant of the
link as well.

\section{Appendix by Tommaso de Fernex: Colength of a product of ideals}

Let $R$ be a local ring with maximal ideal $\fm$, essentially of finite type over a field $k$.
Let $e$ be the embedded dimension of $R$.
For any $\fm$-primary ideal $\fa$, denote by $\ell(R/\fa)$ the length of $R/\fa$.

\begin{lemma}\label{lem:alg}
For any finite collection of $\fm$-primary ideals $\fa_1,\dots,\fa_d \subset R$, we have
\[
d^{e-1} \textstyle{\sum_{i=1}^d} \ell(R/\fa_i) \ge \ell\bigl(R/(\fa_1\cdots\fa_d)\bigr),
\]
and the inequality is strict if $d\ge 2$ and $e \ge 2$.
\end{lemma}

\begin{proof}
By Cohen's structure theorem, there is a surjection $k[[x_1,\dots,x_e]] \to \^R$,
where $\^R$ is the $\fm$-adic completion of $R$. After taking the inverse image of the ideals $\fa_i\^R$
to $k[[x_1,\dots,x_e]]$ and restricting to $k[x_1,\dots,x_e]$, we reduce to prove the lemma when
$R = k[x_1,\dots,x_e]$ and $\fm = (x_1,\dots,x_e)$.
If we fix a monomial order which gives a flat degeneration to monomial ideals,
and denote by $\ini(\fa)$ the initial ideal of an ideal $\fa \subset R$,
then $\ell(R/\fa) = \ell(R/\ini(\fa))$ and
$\prod_{i=1}^d \ini(\fa_i) \subset \ini(\prod_{i=1}^d \fa_i)$.
We can therefore assume that each $\fa_i$ is monomial.

Let $\fa = \prod_{i=1}^d \fa_i$.
For $\bu = (u_1,\dots,u_e) \in \Z^e_{\ge 0}$, we denote $\bx^\bu = \prod_{j=1}^e x_j^{u_j}$.
Let
\[
Q_i = \bigcup_{\bx^\bu \in \fa_i}\(\bu + \R^e_{\ge 0}\)
\quad\text{and}\quad
Q = \bigcup_{\bx^\bu \in \fa}\(\bu + \R^e_{\ge 0}\).
\]
Notice that
$\ell(R/\fa_i) = \Vol\(\R^e_{\ge 0} \smallsetminus Q_i\)$
and
$\ell(R/\fa) = \Vol\(\R^e_{\ge 0} \smallsetminus Q\)$, where the volumes are computed
with respect to the Euclidean metric.
We consider the radial sum
\[
Q' = \mathop{\bigstar}_{i=1}^d Q_i :=
\bigcup_W \sum_{i=1}^d(Q_i \cap W)
\]
introduced in \cite{Fernex}:
the union runs over all rays $W \subset \R^e_{\ge 0}$,
and the sum appearing in the right-hand side is the usual sum of subsets of a vector space.

For every $\bv \in Q'$, we can find $\bv_i \in Q_i$
such that $\bv = \sum_{i=1}^d \bv_i$. For each $i$, we have $\bv_i \in \(\bu_i + \R^e_{\ge 0}\)$
for some $\bu_i \in \Z^e_{\ge 0}$ such that $\bx^{\bu_i} \in \fa_i$.
Then, setting $\bu = \sum_{i=1}^d \bu_i$, we have $\bx^\bu \in\fa$ and
$\bv \in \(\bu + \R^e_{\ge 0}\)$, and therefore $\bv \in Q$.
This means that $Q' \subset Q$, and hence
\begin{equation}\label{eq:1}
\Vol\(\R^e_{\ge 0} \smallsetminus Q'\) \ge \Vol\(\R^e_{\ge 0} \smallsetminus Q\).
\end{equation}
Then, to prove the inequality stated in the lemma, it suffices to show that
\begin{equation}\label{eq:2}
d^{e-1}\(\sum_{i=1}^d \Vol\(\R^e_{\ge 0} \smallsetminus Q_i\)\) \ge \Vol\(\R^e_{\ge 0} \smallsetminus Q'\).
\end{equation}
To this end, we fix spherical coordinates
$(\theta,\r) \in S \times \R_{\ge 0}$ where $S$ is the intersection of the unit sphere with
$\R^e_{\ge 0}$. For any $\theta \in S$, we define
$r_i(\theta) = \inf \{ \r \mid (\theta,\r) \in Q_i \}$
and $r(\theta) = \inf \{ \r \mid (\theta,\r) \in Q' \}$.
By the definition of $Q'$, we have $r(\theta) = \sum_{i=1}^d r_i(\theta)$.
We have
\[
\Vol(\R^e_{\ge 0} \smallsetminus Q_i)
= \int_S \int_0^{r_i(\theta)} \r^{e-1} \,d\r \,\omega(\theta)
=\int_S \frac{r_i(\theta)^e}{e} \,\omega(\theta)
\]
and
\[
\Vol(\R^e_{\ge 0} \smallsetminus Q')
= \int_S\int_0^{r(\theta)}\r^{e-1}\,d\r\,\omega(\theta)
= \int_S \frac{r(\theta)^e}{e}\,\omega(\theta)
\]
for some volume form $\omega$ on $S$.
Then the desired inequality follows from
\begin{equation}\label{eq:3}
d^{e-1}\sum_{i=1}^d r_i(\theta)^e \ge r(\theta)^e,
\end{equation}
which follows from H\"older's inequality.

To conclude, we show that the inequality is strict
if $d \ge 2$ and $e \ge 2$. First observe that \eqref{eq:2}
is a strict inequality unless \eqref{eq:3} is an equality for almost all $\theta \in S$, which
can only happen if $\fa_i = \fa_1$ for every $i$.
Suppose this is the case, so that $\fa = \fa_1^d$.
Notice that in this case $Q'$ is a polyhedron.
Let $a,b$ be the smallest integers
such that $x_1^a \in \fa_1$ and $x_1^{a'}x_2^b \in \fa_1$ for some $a' < a$.
Then $x_1^{(d-1)a + a'}x_2^b \in \fa$, and hence the vector $\bv = ((d-1)a + a',b,0,\dots,0)$ belongs to $Q$.
Note, on the contrary, that $\bv$ is not in $Q'$. Hence $Q' \subsetneq Q$, and since
these sets are polyhedra, it follows that \eqref{eq:1} is a strict inequality.
Therefore the inequality stated in the lemma, which follows as a combination of \eqref{eq:1} and
\eqref{eq:2}, is strict.
\end{proof}

Princeton University, Princeton NJ 08544-1000

\email{kollar@math.princeton.edu}
\medskip

A. R\'enyi Institute of Mathematics, 1053 Budapest, Re\'altanoda u. 13-15, Hungary.

\email{nemethi.andras@renyi.mta.hu}

\medskip
University of Utah, Salt Lake City, UT 08140

\email{defernex@math.utah.edu}

\end{document}